\documentclass[12pt, notitlepage]{amsart}
\usepackage{latexsym, amsfonts, amsmath, amssymb, amsthm, cite}

\newtheorem{theorem}{Theorem}
\newtheorem{lemma}[theorem]{Lemma}
\newtheorem{proposition}[theorem]{Proposition}

\makeatletter
\renewcommand{\pod}[1]{\allowbreak\mathchoice
  {\if@display \mkern 18mu\else \mkern 8mu\fi (#1)}
  {\if@display \mkern 18mu\else \mkern 8mu\fi (#1)}
  {\mkern4mu(#1)}
  {\mkern4mu(#1)}
}
\makeatother

\usepackage{graphicx}
\usepackage{mathrsfs}

\begin{document}
\title{Many solutions to the $S$-unit equation $a+1=c$}
\author{Junsoo Ha and Kannan Soundararajan}
\address{Department of Mathematics, Incheon National University, Incheon, Republic
of Korea}
\email{junsoo.ha.31@gmail.com}
\address{Department of Mathematics, Stanford University, Stanford, CA 94305,
USA}
\email{ksound@stanford.edu}
\thanks{The second author is partially supported by a grant from the National Science Foundation, and 
a Simons Investigator grant from the Simons Foundation.}  
\begin{abstract}
We show that there are arbitrarily large sets 
$S$ of $s$ primes for which the number of solutions to $a+1=c$
where all prime factors of $ac$ lie in $S$ has $\gg \exp( s^{1/4}/\log s)$
solutions.

\end{abstract}

\maketitle

\section{Introduction}

Given a  finite set $S$ of primes, the binary $S$-unit equation concerns solutions to 
$u+v=1$ where $u$ and $v$ are $S$-units; that is, $u$ and $v$ are rational numbers 
whose numerator and denominator are composed only of primes in $S$.   This $S$-unit equation 
has been extensively investigated, and we refer to \cite{EGSTbook} for a detailed overview of this 
equation and its generalizations.    In particular, Evertse \cite{Evertse1984} has shown that the binary $S$-unit 
equation has at most $3 \times 7^{2s+1}$ solutions, where $s$ denotes the cardinality of the set $S$.    
This refines classical work of Siegel which established the finiteness of the number of solutions to the $S$--unit equation.  

While there are many naturally occurring $S$--unit equations that have 
very few solutions (see \cite{EGST1988} for many examples admitting at most two solutions), it 
is possible to exhibit arbitrarily large sets $S$ for which the equation $u+v=1$ has lots of solutions.  
In this context,  Erd{\H o}s, Stewart and Tijdeman \cite{EST1988} showed that
there are arbitrarily large sets $S$ for which the equation has at
least $\exp((4-\epsilon)(s/\log s)^{1/2})$ solutions.   This was subsequently refined by Konyagin and Soundararajan 
\cite{KS2007} who showed that there are sets $S$ for which the $S$--unit equation has at least $\exp(s^{2-\sqrt{2}-\epsilon})$ 
solutions.   The sets $S$ used in these constructions are special and comprise of the set of initial primes, together with a 
small number of primes that appear in the argument, and which are out of our control.  But even for the set $S$ comprising of the 
first $s$ primes, it is expected that the $S$ unit equation has $\exp(s^{2/3-\epsilon})$ solutions, and that perhaps the general $S$ 
unit equation does not have more than $\exp(s^{2/3+\epsilon})$ solutions (see \cite{EST1988} for a heuristic discussion).  
In the context of $S$ being the first $s$ primes (which is related to the distribution of smooth numbers), Lagarias and Soundararajan 
\cite{LS2012} showed that under the Generalized Riemann Hypothesis one has at least $\exp(s^{1/8-\epsilon})$ solutions, and Harper \cite{Harper2014} has 
shown unconditionally that there are at least $\exp(s^{\delta})$ solutions for some $\delta>0$.   Ha \cite{Ha} has studied the analogous problem 
over function fields, obtaining unconditionally $\gg \exp(s^{1/6-\epsilon})$ solutions.

Rewrite the $S$-unit equation $u+v=1$ as $a+b=c$ where $a$, $b$ and $c$ are coprime positive integers with 
all prime factors of $abc$ lying in the set $S$.   In this setting, we may consider the special case when $b=1$, where 
we are seeking two consecutive natural numbers $a$ and $c$ with all their prime factors lying in $S$.  Konyagin and Soundararajan 
\cite{KS2007} showed that this special case too has exponentially many solutions for certain well chosen sets $S$.   Namely, they showed that 
there are sets $S$ for which the equation $a+1=c$ has at least $\exp(s^{1/16})$ solutions.   This was subsequently improved by Harper \cite{Harper2011a} 
who showed the existence of sets $S$ for which there are at least $\exp(s^{1/6-\epsilon})$ solutions.  In this paper we make further progress on this 
question, by showing that there are sets $S$ with at least $\exp(s^{1/4}/\log s)$ solutions.

 \begin{theorem}
\label{thm:first} For all $s$, there exist sets $S$ of $s$
primes such that the equation 
\[
a+1=c
\]
 has $\gg \exp(s^{1/4}/\log s)$ solutions where
all prime factors of $ac$ lie in $S$. 
\end{theorem}

For the equation $a+1= c$, we do not know any upper bound on the number of solutions better than Evertse's bound for 
the more general equation $a+b=c$.   One may also ask for analogues of the results of Lagarias and Soundararajan, and Harper, 
where $S$ is taken to be the set of first $s$ primes.  This remains unknown, but heuristic considerations (as in \cite{LS2012} and \cite{EST1988}) 
suggest that when $S$ is the set of first $s$ primes there are $\exp(s^{1/2-\epsilon})$ solutions to the equation $a+1=c$, and that for general sets $S$ 
the equation has no more than $\exp(s^{1/2+\epsilon})$ solutions.  
 
\section{Deducing Theorem 1 from the main proposition}

In this section we enunciate the main technical result of the paper, from which we shall 
deduce Theorem 1.    Let $y$ be large, and let $\ell \le k$ be two integer parameters.  
Our goal is to evaluate asymptotically 
\begin{equation} 
\label{2.1} 
{\mathcal N}(y;k, \ell) = \# \{ p_1 \cdots p_k \equiv 1 \pmod{q_1\cdots q_{\ell}} \}, 
\end{equation}
where the $p_i$ run over all primes in the interval $(y/2,y]$ and the $q_j$ run over all primes in the interval $(y/4,y/2]$.   
For brevity, we write
\begin{equation} 
\label{2.2} 
\lambda=\sum_{y/4 <q \le y/2}\frac{1}{q}\sim\frac{\log2}{\log y}, 
\end{equation} 
and
\begin{equation} 
\label{2.3} 
P=\sum_{y/2 < p \le y} 1 \sim\frac{y}{2\log y}. 
\end{equation} 
We have in mind ranges where $k$ and $\ell$ grow with $y$, and in the estimates 
below all implied constants will be absolute.  

\begin{proposition}
\label{prop2.1}   Let $y\geq 10$ be a real number, and let $\ell$, $k$ be integers
with $1 \leq \ell \leq k\leq y^{1/3}/(\log y)^{2}$. In the range $\ell \le k/2$ we have 
\[
{\mathcal N}(y;k,\ell) =\lambda^{\ell}P^{k}\Big(1+O\Big(\frac{1}{\log y}\Big)\Big).  
\]
In the range $k/4\leq \ell \leq k/2$, we have 
\[
{\mathcal N}(y;k,\ell) =\lambda^{\ell}P^{k}\Big(1+O\Big(\frac{1}{\log y}\Big)\Big)+O\Big( \ell^{k-\ell} (4\lambda P)^\ell y^{k/2} \Big).  
\]
\end{proposition}

Roughly speaking, Proposition \ref{prop2.1} may be viewed as an average result on the equidistribution of smooth numbers 
in arithmetic progressions.  In this sense, it is related to recent results of Harper \cite{Harper2012} and Drappeau \cite{Drappeau2015} 
which establish strong analogues of the Bombieri--Vinogradov theorem in this context.   For our application to Theorem 1, we are essentially 
interested in the distribution in progressions of integers $n\le x$ that are $(\log x)^{4}$ smooth.  The results of Drappeau would permit 
a larger level of distribution in terms of the moduli of the progressions involved, but they require  a smoothness  of $(\log x)^A$ for a 
suitably large unspecified constant $A$, and therefore are not immediately applicable to our situation.   

\begin{proof}[Proof of Theorem 1]  Put $\ell = \alpha k$ and $k = y^{\beta}/(10 \log y)$, with $0\le \alpha \le 1/2$ and 
$\beta \le 1/3 - \log \log y/\log y$.   With a little calculation using Proposition \ref{prop2.1} we see that if $(1-\alpha)(1-\beta) \ge 1/2$ then  
the error term in the second assertion of the proposition is negligible compared to the main term, and we have 
$$ 
{\mathcal N}(y;k,\ell) = \lambda^\ell P^k \Big( 1+ O\Big( \frac{1}{\log y}\Big)\Big) \ge \frac 12 \lambda^\ell P^k. 
$$

Let ${\mathcal Q}$ denote the set of numbers composed of exactly $\ell$ primes taken from $(y/4,y/2]$ and 
${\mathcal R}$ denote the set of numbers composed of exactly $k$ primes taken from $(y/2,y]$.   We consider 
solutions to the congruence $r\equiv 1 \pmod q$ with $r\in {\mathcal R}$ and $q\in {\mathcal Q}$.   Each solution 
is counted at most $k! \ell!$ times in ${\mathcal N}(y;k,\ell)$, and therefore the number of solutions to this congruence is 
 at least $\frac 12 \lambda^{\ell}P^k/(k! \ell!)$.  For a solution to $r\equiv 1\pmod q$, note that $u=(r-1)/q$ is an integer 
 lying below $y^{k}/(y/4)^{\ell}  = 4^{\ell} y^{k-\ell}$.   It follows that there is a ``popular" integer $u_0$ such that the equation 
 $r=1+qu_0$ has at least 
 $$ 
 \frac 12 \frac{\lambda^{\ell}P^k}{k! \ell!} \frac{1}{4^{\ell} y^{k-\ell}} \gg \Big(\frac{1}{4\ell \log y}\Big)^{\ell} \Big( \frac{y}{k\log y}\Big)^k y^{\ell-k} \gg 10^k y^{-k\beta + (1-\beta)\ell} 
 $$ 
solutions.  If $\alpha (1-\beta) \ge \beta$, then this number of solutions exceeds $10^k$.  

The two constraints $(1-\alpha)(1-\beta) \ge 1/2$ and $\alpha(1-\beta) \ge \beta$ are met by taking $\beta= 1/4$, and $\alpha= 1/3$.  Take 
$S$ to be the set of primes in $(y/4,y]$ together with the prime factors of $u_0$.   Since $u_0$ has at most $\ll (\log u_0)/\log \log u_0 \ll y^{\beta}$ distinct prime factors, 
the set $S$ has size at most $y/\log y$.  Our argument above has produced 
$$ 
\gg 10^k \gg \exp\Big( \frac{y^{\beta}}{5\log y}\Big) \ge \exp\Big( \frac{s^{1/4}}{10 (\log s)^{3/4}}\Big) 
$$ 
solutions to the equation $a+1=c$ with all prime factors of $ac$ lying in $S$.   This establishes the theorem.  
\end{proof}

\section{Proof of Proposition \ref{prop2.1}}  

By the orthogonality relation for Dirichlet characters, we have 
\begin{align}
{\mathcal N}(y;k,\ell) &= \sum_{\substack{y/4 <q_j \le y/2 \\ 1\le j \le \ell }} \frac{1}{\varphi(q_1 \cdots q_{\ell})} \sum_{\chi \pmod {q_1\cdots q_{\ell}}} 
\sum_{\substack{y/2 <p_i \le y \\ 1\le i \le k} } \chi(p_1 \cdots p_k) \nonumber\\ 
&=  \sum_{\substack{y/4 <q_j \le y/2 \\ 1\le j \le \ell }} \frac{1}{\varphi(q_1 \cdots q_{\ell})} \sum_{\chi \pmod {q_1\cdots q_{\ell}}} \Big( \sum_{y/2 < p \le y}\chi(p)\Big)^k. 
\label{3.1} 
\end{align}
We isolate the contribution of the principal character $\chi = \chi_0$ above.  Since $\varphi(q_1 \cdots q_\ell) = q_1 \cdots q_\ell (1+O(\ell/y))$, this term contributes 
\begin{equation} 
\label{3.2}
 \Big(1 + O\Big(\frac{\ell}{y}\Big)\Big) \sum_{\substack{y/4 <q_j \le y/2 \\ 1\le j \le \ell }} \frac{1}{q_1 \cdots q_{\ell}} \Big(\sum_{y/2 < p \le y} 1\Big)^k 
 = \Big( 1+ O\Big( \frac{\ell}{y}\Big) \Big) \lambda^\ell P^k. 
 \end{equation} 
 It remains now to estimate the contribution of the non-principal characters to (\ref{3.1}), which is bounded by
 \begin{equation} 
 \label{3.3} 
 \le \sum_{\substack{y/4 <q_j \le y/2 \\ 1\le j \le \ell }} \frac{2}{q_1 \cdots q_{\ell}} \sum_{\substack{\chi \pmod {q_1\cdots q_{\ell} } \\ \chi \neq \chi_0}}
  \Big| \sum_{y/2 < p \le y}\chi(p)\Big|^k.
 \end{equation} 
 
 To estimate the contribution of the non-principal characters, we shall use the large sieve.  Since the large sieve gives a bound for sums over 
 primitive characters, we first transform (\ref{3.3}) into a sum over primitive characters.  Recall that each non-principal character $\chi \pmod{q_1 \cdots q_{\ell}}$ 
 is induced by some primitive character $\widetilde{\chi} \pmod{q}$ where $q>1$ is a divisor of $q_1\cdots q_{\ell}$.   For integers $1\le t\le \ell$ define 
 ${\mathcal Q}_t$ to be the set of moduli $q$ that are composed of exactly $t$ primes (not necessarily distinct) all taken from the interval $(y/4,y/2]$.   Thus 
 the sum in (\ref{3.3}) may be recast as 
 $$ 
 \sum_{t=1}^{\ell} \ \   \sum_{q \in {\mathcal Q}_t} \ \ \sideset{}{^{\star}} \sum_{\substack{\widetilde{\chi} \pmod q }} \ \ \ 
 \Big( \sum_{ \substack {y/4 <q_j \le y/2 \\ 1\le j \le \ell \\ q| q_1\cdots q_\ell} } \frac{2}{q_1 \cdots q_{\ell} } \Big)  \Big| \sum_{y/2<p\le y} \widetilde{\chi}(p) \Big|^k.
 $$ 
  Here the $\star$ indicates that the sum is over primitive characters, and we used that $\chi(p) = \widetilde{\chi}(p)$ for $y/2 < p\le y$.   Given 
  $q\in {\mathcal Q}_t$ note that 
  $$
\sum_{ \substack {y/4 <q_j \le y/2 \\ 1\le j \le \ell \\ q| q_1\cdots q_\ell} } \frac{2}{q_1 \cdots q_{\ell} }  \le \frac{2}{q} \binom{\ell}{t} t! \Big(\sum_{y/4 < p \le y/2}\frac 1p\Big)^{\ell-t} 
= \frac{2}{q} \frac{\ell!}{(\ell-t)!} \lambda^{\ell-t},
  $$ 
since we must pick $t$ out of $q_1$, $\ldots$, $q_{\ell}$ to be the $t$ prime factors of $q$, and these $t$ prime factors may be permuted 
in at most $t!$ ways.   Since $\ell!/(\ell-t)! \le \ell^t$ and $q\ge (y/4)^t$ for $q\in {\mathcal Q}_t$, we conclude that the quantity in (\ref{3.3}) may be bounded by 
\begin{equation} 
\label{3.4} 
\ll \sum_{t=1}^{\ell} \Big( \frac{4\ell}{y}\Big)^t \lambda^{\ell -t} \sum_{q\in {\mathcal Q}_t} \ \ \ 
\sideset{}{^{\star}} \sum_{\substack{\widetilde{\chi} \pmod q }} \ \ \  \Big|\sum_{y/2 <p\le y} 
\widetilde{\chi}(p) \Big|^k. 
\end{equation} 

We are now ready to apply the large sieve, which we now recall. 

\begin{lemma} 
\label{lem:large.sieve} For any sequence $a_{n}$ of complex numbers, we
have
\begin{equation}
\label{3.5} 
\sum_{\chi\pmod q}\Big|\sum_{n\leq N}a_{n}\chi(n)\Big|  \leq(N+q)\sum_{n\leq N}|a_{n}|^{2}, 
\end{equation}
and 
\begin{equation} 
\label{3.6} 
\sum_{q\leq Q}\frac{q}{\varphi(q)} \ \ \sideset{}{^{\star}}\sum_{\chi\pmod q}\Big|\sum_{n\leq N}a_{n}\chi(n)\Big|^{2}\leq\left(N+Q^{2}-1\right)\sum_{n\leq N}\left|a_{n}\right|^{2}. 
\end{equation}
\end{lemma}

\begin{proof}  Estimate \eqref{3.5} follows from the orthogonality of Dirichlet characters, while \eqref{3.6} may be found,
for example, in Theorem 7.13 of \cite{IwKo}.
\end{proof}

From the large sieve we extract two bounds related to the quantity \eqref{3.4}: namely, 
\begin{equation} 
\label{3.7} 
\sum_{ q\in {\mathcal Q}_t}      \ \ \ 
\sideset{}{^{\star}} \sum_{\substack{\widetilde{\chi} \pmod q }} \ \ \  \Big|\sum_{y/2 <p\le y} 
\widetilde{\chi}(p) \Big|^{2t} \ll  y^t P^{2t}, 
\end{equation} 
and 
\begin{equation} 
\label{3.8} 
\sum_{ q\in {\mathcal Q}_t}      \ \ \ 
\sideset{}{^{\star}} \sum_{\substack{\widetilde{\chi} \pmod q }} \ \ \  \Big|\sum_{y/2 <p\le y} 
\widetilde{\chi}(p) \Big|^{4t} \ll  (tyP)^{2t}. 
\end{equation} 

Consider first the estimate \eqref{3.7}.   Write $(\sum_{y/2 < p \le y} \widetilde{\chi}(p))^t = \sum_{n\le y^{t}} a_t(n) \widetilde{\chi}(n)$, 
where $a_t(n)$ denotes the number of ways of writing $n$ as a product of $t$ primes all from the interval $(y/2,y]$.   Clearly 
$a_t(n) \le t!$ and $\sum_n a_t(n) = P^t$.   Therefore, using the large sieve estimate \eqref{3.5} we find 
$$ 
\sum_{ q\in {\mathcal Q}_t}      \ \ \ 
\sideset{}{^{\star}} \sum_{\substack{\widetilde{\chi} \pmod q }} \ \ \  \Big|\sum_{y/2 <p\le y} 
\widetilde{\chi}(p) \Big|^{2t} \ll |{\mathcal Q}_t| y^t \sum_{n\le y^t} a_t(n)^2 \le |{\mathcal Q}_t| y^t t! P^t.
$$ 
Since $t\le \ell \le y^{1/3}$ it is easy to check that $|{\mathcal Q}_t| \le P^t/t!$ for large $y$, and therefore 
\eqref{3.7} follows.  

 The proof of \eqref{3.8} is similar, invoking now the large sieve estimate \eqref{3.6}.  With $a_{2t}(n)$ defined 
 similarly as above, \eqref{3.6} yields 
 $$
 \sum_{ q\in {\mathcal Q}_t}      \ \ \ 
\sideset{}{^{\star}} \sum_{\substack{\widetilde{\chi} \pmod q }} \ \ \  \Big|\sum_{y/2 <p\le y} 
\widetilde{\chi}(p) \Big|^{4t} \ll  y^{2t} \sum_{n\le y^{2t}} a_{2t}(n)^2  \le y^{2t} (2t)! P^{2t}, 
$$ 
from which \eqref{3.8} follows.

 If $k\ge 4t$ then from \eqref{3.8} and the trivial bound $|\sum_{y/2 < p \le y} \widetilde{\chi}(p)| \le P$ we get 
 $$ 
 \Big( \frac{4\ell}{y} \Big)^t \lambda^{\ell-t}  \sum_{ q\in {\mathcal Q}_t}      \ \ \ 
\sideset{}{^{\star}} \sum_{\substack{\widetilde{\chi} \pmod q }} \ \ \  \Big|\sum_{y/2 <p\le y} 
\widetilde{\chi}(p) \Big|^{k} \ll \Big(\frac{4\ell}{y}\Big)^t \lambda^{\ell-t} P^{k-4t} (tyP)^{2t} = P^k \lambda^{\ell} \Big( \frac{4\ell t^2 y}{\lambda P^2}\Big)^t.
$$ 
Since we are assuming that $\ell \le k \le y^{1/3}/(\log y)^2$, we may conclude that  
\begin{equation}
\label{3.9} 
\sum_{1\le t\le k/4} \Big( \frac{4\ell}{y} \Big)^t \lambda^{\ell-t}  \sum_{ q\in {\mathcal Q}_t}      \ \ \ 
\sideset{}{^{\star}} \sum_{\substack{\widetilde{\chi} \pmod q }} \ \ \  \Big|\sum_{y/2 <p\le y} 
\widetilde{\chi}(p) \Big|^{k} \ll \sum_{1\le t\le k/4} P^k \lambda^\ell (\log y)^{-t} \ll \frac{P^{k} \lambda^{\ell}}{\log y}. 
\end{equation}

Now suppose $k/4\leq t\leq k/2$.   Interpolating between \eqref{3.7} and \eqref{3.8} using H{\" o}lder's inequality we obtain 
$$ 
 \sum_{ q\in {\mathcal Q}_t}      \ \ \ 
\sideset{}{^{\star}} \sum_{\substack{\widetilde{\chi} \pmod q }} \ \ \  \Big|\sum_{y/2 <p\le y} 
\widetilde{\chi}(p) \Big|^{k} \ll  \big( y^t P^{2t} \big)^{\frac{4t-k}{2t}} \big( (tyP)^{2t} \big)^{\frac{k-2t}{2t}} = t^{k-2t} P^{2t} y^{k/2}.
$$ 
Therefore, for $\ell \le k/2$, 
\begin{align*}
\sum_{k/4<t\le \ell} \Big(\frac 4y\Big)^t \lambda^{\ell -t}  \sum_{ q\in {\mathcal Q}_t}      \ \ \ 
\sideset{}{^{\star}} \sum_{\substack{\widetilde{\chi} \pmod q }} \ \ \  \Big|\sum_{y/2 <p\le y} 
\widetilde{\chi}(p) \Big|^{k} & \ll   \sum_{k/4<t\le \ell} \ell^k \lambda^\ell y^{k/2} \Big( \frac{4P^2}{\lambda \ell y} \Big)^t 
\\
&\ll \ell^k \lambda^\ell y^{k/2} \sum_{k/4 <t \le \ell} \Big( \frac{4P}{\ell}\Big)^t.
\end{align*}
The right side above is dominated by the term $t=\ell$, and so we conclude that 
\begin{equation} 
\label{3.10} 
\sum_{k/4<t\le \ell} \Big(\frac 4y\Big)^t \lambda^{\ell -t}  \sum_{ q\in {\mathcal Q}_t}      \ \ \ 
\sideset{}{^{\star}} \sum_{\substack{\widetilde{\chi} \pmod q }} \ \ \  \Big|\sum_{y/2 <p\le y} 
\widetilde{\chi}(p) \Big|^{k}  \ll \ell^{k-\ell} (4\lambda P)^{\ell} y^{k/2}. 
\end{equation} 
 
 The estimates \eqref{3.9} and \eqref{3.10} complete the proof of the proposition. 

\bibliographystyle{plain}
\bibliography{harper}

\end{document}